\documentclass{article}
\usepackage{amsmath}
\usepackage{amsthm}
\usepackage{pgf,tikz}
\usepackage{mathrsfs}
\usetikzlibrary{arrows}
\usepackage{cite}
\usepackage{float}
\usepackage{algorithmicx}
\usepackage[ruled,vlined]{algorithm2e}
\usepackage[noend]{algpseudocode}
\usepackage{graphicx}
\usepackage{subfig}
\usepackage{verbatim}
\usepackage{makeidx}
\usepackage{showidx}
\usepackage{amssymb}
\usepackage{hyperref}
\makeatletter
\def\BState{\State\hskip-\ALG@thistlm}
\makeatother

\newtheorem{theorem}{Theorem}

\newtheorem{lemma}{Lemma}

\newtheorem{prop}{Proposition}
\newtheorem{ass}{Assumption}

\begin{document}
	\title{The convergence of the Regula Falsi method}
	\author{ Trung Nguyen\\
	{\footnotesize VinAI Research, Vietnam}\\
	{\footnotesize nguyenductrung.samihust@gmail.com}
	}
	
	\date{}
	\maketitle 
	\begin{abstract}
	     Regula Falsi, or the method of false position, is a numerical method for finding an approximate solution to \(f(x)=0\) on a finite interval \([a,b]\), where \(f\) is a real-valued continuous function on \([a,b]\) and satisfies \(f(a)f(b)<0\). Previous studies proved the convergence of this method under certain assumptions about the function \(f\), such as, both the first and second derivatives of $f$ do not change sign on [\(a,b\)]. In this paper, we remove those assumptions and prove the convergence of the method for all continuous functions.
	\end{abstract}
\noindent
\section{Introduction}
Regula Falsi method appears in many undergraduate numerical analysis documents and books (e.g., \cite{anh,BF,G,sb,SR,vinh,w}). In \cite{anh,vinh}, this method is also called the \emph{chord method}. The goal of this method is to generate a sequence \(\{x_n\}_{n=0}^{\infty}\) that converges to a zero \(x^*\) of \(f\) when \(n\) tends to infinity. 
Some provided the proofs under some assumptions, such as both the first and second derivatives of \(f\) do not change sign on [\(a,b\)] (e.g., \cite{anh,G,sb,vinh}).
In this paper, we remove those assumptions and prove the convergence of the method for all continuous functions.

\section{Main result}
In this paper, the symbol \(\mathbb{N}\) denotes the set of natural numbers \(\{0,1,2,\ldots\}\) and \(\mathbb{N^*}=\mathbb{N}\backslash\{0\}\).
Let \(f: [a,b]\rightarrow\mathbb{R}\) be a continuous function on the interval \([a,b]\) and \(f(a)f(b)<0\).  
We construct a sequence \(\{x_n\}_{n=0}^{\infty}\) as follows.

\begin{algorithm}[H] 
	\caption{Regula Falsi}
	\label{al1}
Set	$u_1:=a$,\ $v_1:=b$. \\
\For{$n=1,2,3,\ldots$}{
	\begin{equation}
	\label{1}
	{x_n:=\frac{u_nf(v_n)-v_nf(u_n)}{f(v_n)-f(u_n)}}\text{ .}
	\end{equation}
	\If {$f(x_n)f(u_n)>0$}{\(u_{n+1}:=x_n,\) \(v_{n+1}:=v_n\)}
	\Else {$u_{n+1}:=u_n$, $ v_{n+1}:=x_n$}
}
For the completeness, we will set $x_0$ as follows:\\
\If{$f(x_1)f(a)>0$}{\(x_0:=b\)} \Else{$x_0:=a$.} 
\end{algorithm}

\vskip 0.2cm 

The main aim of this paper is to prove the following result.

\begin{theorem}\label{thm1} For any continuous function $f$ on a closed, bounded interval $[a,b]$ that satisfies $f(a)f(b)<0$, the sequence \(\{x_n\}_{n=0}^{\infty}\) generated by Algorithm \ref{al1} converges to a zero \(x^*\) of \(f\), i.e., \(f(x^*)=0\).
\end{theorem}
The proof of this theorem is lengthy, and hence, it will be divided into several parts. Clearly, if \(f(x_{n_0})=0\) for some \(n_0\in \mathbb{N}\), then simple induction implies that \(x_n=x_{n_0}\) for all \(n\geq n_0\) and the desired convergence is obvious. Now, we consider the case where $f(x_n)\not=0$ for all $n$.  

\begin{lemma}\label{lemma 1}
Considering the sequence \(\{r_n\}_{n=1}^{\infty}\) defined by
\begin{equation}
\label{2}
r_n:=\frac{v_{n+1}-u_{n+1}}{v_n-u_n} \ .
\end{equation}
Then for each \(n\), we have that:
\begin{equation}\label{3}
r_n = 
\begin{cases}
 & \cfrac{1}{1+|\frac{f(u_n)}{f(v_n)}|} \ \mbox{ if } u_{n+1}=x_n \mbox{ and } v_{n+1}=v_n,  \\
 & \cfrac{1}{1+|\frac{f(v_n)}{f(u_n)}|} \ \mbox{ if } u_{n+1}=u_n \mbox{ and } v_{n+1}=x_n. 
 \end{cases}
\end{equation}
\end{lemma}	

\begin{proof} Without loss of generality, we consider only the case where \(u_{n+1}=u_n\) and \(v_{n+1}=x_n\). Thus, we have
\[	v_{n+1}-u_{n+1}=x_n-u_n=\frac{u_nf(v_n)-v_nf(u_n)}{f(v_n)-f(u_n)}-u_n=\frac{(u_n-v_n)f(u_n)}{f(v_n)-f(u_n)}.
\]
It follows by \((\ref{2})\) that 
\[	r_n=\frac{-f(u_n)}{f(v_n)-f(u_n)}=\frac{1}{1-\frac{f(v_n)}{f(u_n)}}=\frac{1}{1+|\frac{f(v_n)}{f(u_n)}|}\text{ ,}
\]
which is \eqref{3} exactly.\\
Analoguously, in the second case, where \(u_{n+1}=x_n\) and \(v_{n+1}=v_n\), we have that
\begin{equation*}
\label{6}
r_n=\frac{f(v_n)}{f(v_n)-f(u_n)}=\frac{1}{1-\frac{f(u_n)}{f(v_n)}}=\frac{1}{1+|\frac{f(u_n)}{f(v_n)}|}\ . 
\end{equation*}
This completes the proof of Lemma 1.
\end{proof}

\begin{lemma}
	\label{lemma2}
	For each fixed \(n_0\in\mathbb{N^*}\), the sequence \(\{x_n\}_{n\geq n_0}\) is either upper or lower bounded by \(x_{n_0}\). 
\end{lemma}
\begin{proof}
Firstly, we notice the decreasing property of the sequence \(\{[u_n,v_n]\}_{n=1}^{\infty}\), i.e.
\[
[u_1,v_1]  \supsetneq  [u_2,v_2] \supsetneq \ldots \supsetneq [u_n,v_n] \supsetneq \ldots
\]
In addition, since for any $n$ the point $x_n$ is either $u_{n+1}$ or $v_{n+1}$, the desired claim follows directly.
\end{proof}

Due to Lemma \ref{lemma 1}, we have that \(	0<r_n<1 \text{ for all } n \in \mathbb{N^*}.\)
Let \(r=\sup\limits_{n}r_n\), we see that \(0<r\leq1\). Now we discuss two cases whether \(r\) is strictly smaller than 1 or not in the following proposition.\\

\begin{prop}\label{Pro1}
Consider the sequences \(\{x_n\}_{n=0}^{\infty}\) generated by Algorithm \ref{al1} and \(\{r_n\}_{n=0}^{\infty}\) defined by \eqref{2}, \eqref{3}. Then the following claims hold true.\\
i) If \(\sup\limits_{n}r_n < 1\) then the sequence \(\{x_n\}_{n=0}^{\infty}\) converges to a zero of \(f\). \\
ii) If \(\sup\limits_{n}r_n = 1\) then there exists a subsequence \(\{x_{k_n}\}_{n=0}^{\infty}\) converges to a zero of \(f\). \\
\end{prop}
\begin{proof}
\noindent i) If \(r<1\) then \(v_{n+1}-u_{n+1}\leq r(v_n-u_n)\leq \ldots\leq r^n(v_1-u_1)\to0\) when $n$ tends to infinity. This implies that \(\{[u_n,v_n]\}_{n=1}^{\infty}\) is a decreasing nested sequence of non-empty, closed and bounded intervals in $\mathbb{R}$ and their lengths strictly decrease to $0$. So, due to Cantor intersection theorem, see e.g. \cite{Kor}, there exists \(x^*\in[a,b]\) such that $\{x^*\}=\cap_{n}[u_n,v_n]$. Evidently, $\lim\limits_{n\to\infty}x_n=\lim\limits_{n\to\infty}u_n=\lim\limits_{n\to\infty}v_n=x^*$. Since \({f(u_n)f(v_n)<0}\) for all \(n\), due to the continuity of \(f\), let \(n\to\infty\) then \(f(x^*)f(x^*)\leq0\), which only holds if \(f(x^*)=0\). Consequently, the sequence \(\{x_n\}_{n=0}^{\infty}\) generated by Algorithm \ref{al1} converges to a zero $x^*$ of \(f\).\\

\noindent ii) Now we consider the case where \(r=1\). Firstly, we observe that by the construction of the sequence \(\{x_n\}_{n=0}^{\infty}\), for each \(n\in \mathbb{N^*}\), there exist \(p,q \in \mathbb{N}\) depending on \(n\) such that \(u_n=x_p\) and \(v_n=x_q\). Since \(\sup\limits_{n}r_n=1\), there exists a subsequence $\{r_{m_n}\}_{n=1}^{\infty}$ of \(\{r_n\}_{n=1}^{\infty}\) that converges to 1. Let \(M=\max\limits_{x\in[a,b]}|f(x)|\) then the continuity of $f$ implies that \(M<\infty\). Due to Lemma \ref{lemma 1} we have that 
\begin{equation}
\label{8}
r_{m_n} < \frac{1}{1+\frac{|f(x_{m_n})|}{M}} < 1  \mbox{ for all } n\in \mathbb{N^*}.
\end{equation} 
Therefore, let \(n\) tend to infinity we have $\underset{n\rightarrow \infty}{\lim}f(x_{m_n})=0$. Again, Bolzano-Weierstrass theorem applied to the bounded sequence \(\{x_{m_n}\}_{n=1}^{\infty}\) implies the existence of a convergent subsequence \(\{x_{k_n}\}_{n=0}^{\infty}\), whose limit is denoted by \(x^*\). The continuity of $f$, therefore, implies that $f(x^*)=0$. 
\end{proof}

Clearly, from Proposition \ref{Pro1}, the remaining case is where $r=1$, and we need to prove the convergence of the whole sequence \(\{x_n\}_{n=0}^{\infty}\) instead of only the subsequence \(\{x_{k_n}\}_{n=0}^{\infty}\). We will perform this task by disproof, i.e., the following assumption holds.
\begin{ass}\label{ass} Suppose that the sequence \(\{x_n\}_{n=0}^{\infty}\) generated by Algorithm \ref{al1} diverges.
\end{ass}
Under Assumption \ref{ass}, there exists \(\epsilon>0\) and a subsequence \(\{x_{i_n}\}_{n=0}^{\infty}\) of \(\{x_n\}_{n=0}^{\infty}\) such that
\begin{equation}
\label{9}
	|x_{i_n}-x^*|>\epsilon\text{ for all }n\geq0.
\end{equation}
Since the sequence \(\{x_{k_n}\}_{n=0}^{\infty}\) converges to \(x^*\), there exists \(n_0\in \mathbb{N}\) such that
\begin{equation}
\label{10}
|x_{k_n}-x^*|<\epsilon \text{ for all } n\geq n_0.
\end{equation}
By suitably removing a finite number of starting elements of two above sequences, one can assume that \(n_0=0\) and \(i_0>k_0\).

\begin{lemma}\label{lemma3}
i) The point \(x^*\) must lie between \(x_{k_0}\) and \(x_{i_0}\).\\
ii) Two sequences \(\{x_{k_n}\}_{n=0}^{\infty}, \{x_{i_n}\}_{n=0}^{\infty}\) are 2 monotone sequences in opposite directions. 
\end{lemma}
\begin{proof}
i) Due to (\ref{9}) and (\ref{10}), we first observe that \(x_{i_0}\) could not lie between \(x_{k_0}\) and \(x^*\).
In case \(x_{k_0}\) is between \(x^*\) and \(x_{i_0}\), we suppose that \(x^*<x_{k_0}<x_{i_0}\), since similar arguments can be made for \(x_{i_0}<x_{k_0}<x^*\). Due to Lemma \ref{lemma2}, \(x_n>x_{k_0}\) for all \(n>k_0\). 
Thus, there does not exist any subsequence of \(\{x_n\}_{n=0}^{\infty}\) converging to \(x^*\), which contradicts the fact that $\lim\limits_{n\to\infty}x_{k_n}=x^*$. Therefore, \(x^*\) must lie between \(x_{k_0}\) and \(x_{i_0}\). \\
ii) Without loss of generality, we suppose that \(x_{k_0}>x^*>x_{i_0}\). The case where \(x_{k_0}<x^*<x_{i_0}\) is on the analogy. Since $\lim\limits_{n\to\infty}x_{k_n}=x^*$, there exists a \(k_{n_1}\) such that $k_{n_1}>i_0>k_0$ and \(x_{k_0}>x_{k_n}>x_{i_0}\). Due to Lemma \ref{lemma2} we obtain that \(\{x_n\}_{n>i_0}\) is upper bounded by \(x_{k_0}\) and lower bounded by \(x_{i_0}\). 
For each fixed \(k_n\), if there exists an \(i_m\) such that \(i_m>k_n\) and \(x_{i_m}>x_{k_n}\) then $x_{k_0}>x_{i_m}>x_{k_n}$, and hence,  \(|x^*-x_{i_m}|<\text{max}\{|x^*-x_{k_n}|,|x^*-x_{k_0}|\}<\epsilon\), which contradicts (\ref{9}). Thus, \(x_{i_m}<x_{k_n}\) for all \(i_m>k_n\).
This, due to by Lemma (\ref{lemma2}), follows that \(x_{j}<x_{k_n}\) for all $j>k_n$, and hence, \(x_{k_{n+1}}<x_{k_n}\). In other words, the sequence \(\{x_{k_n}\}_{n=0}^{\infty}\) is strictly decreasing.
Similarly, the sequence \(\{x_{i_n}\}_{n=0}^{\infty}\) is strictly increasing.
\end{proof}

Due to Lemma \ref{lemma3}, without loss of generality, we suppose that \(x_{k_0}>x^*>x_{i_0}\). In this case, the sequence \(\{x_{k_n}\}_{n=0}^{\infty}\) is strictly decreasing and the sequence \(\{x_{i_n}\}_{n=0}^{\infty}\) is strictly increasing. 
Furthermore, their boundedness implies that they converge to some limits. As we know, \(x^*=\lim\limits_{n\to\infty}x_{k_n}\). 
Let us denote by \(y\) the limit of \(\{x_{i_n}\}_{n=0}^{\infty}\). Following from \eqref{9}, for all $n$, we have $x_{i_n} < x^*-\epsilon$, and hence, \(y \leq x^*-\epsilon <x^*\).
 
\begin{lemma}\label{lemma4}
i) There does not exist any \(x_n\) such that \(y\leq x_n\leq x^*\). \\
ii) The subsequences of \(\{x_n\}_{n=0}^{\infty}\) belonging to the intervals \([x_{i_0},y)\) and \((x^*,x_{k_0}]\) are the strictly monotone. Moreover, they converge to \(y\) and \(x^*\) respectively. 
\end{lemma}
\begin{proof}
i) Suppose the contrary, i.e., there exists \(x_n\) such that \(y\leq x_n\leq x^*\). Then, there exists \( m\in \mathbb{N}\) such that \(i_m>n,k_m>n\) and \(x_{i_m}<x_n<x_{k_m}\), which contradicts Lemma \ref{lemma2}. 

ii) Firstly, we consider the subsequence of \(\{x_n\}_{n=0}^{\infty}\) lying within \([x_{i_0},y)\). Because \(\lim\limits_{m\to\infty}x_{i_m}=y\), for each such \(x_n\), there exists \(m\in\mathbb{N}\) such that \(x_n<x_{i_m}<y\). Lemma \ref{lemma2} implies that \(x_k>x_n\) for all \(k>n\). This leads to the monotonicity of the subsequence being considered. Consequently, this subsequence converges to \(y\).
	Similarly, the subsequence lying within \((x^*,x_{k_0}]\) is strictly decreasing and converges to \(x^*\).  
\end{proof}

Making use of Lemma \ref{lemma4}, for notational convernience, from now we use the notations \(\{x_{i_n}\}_{n=0}^{\infty}\) and \(\{x_{k_n}\}_{n=0}^{\infty}\) to denote the subsequences of \(\{x_n\}_{n=0}^{\infty}\) lying within \([x_{i_0},y)\) and \((x^*,x_{k_0}]\), respectively. 
Notice that, due to Lemma \ref{lemma4}, we have 
$$\{x_n\}_{n=i_0}^{\infty} = \{x_{i_n}\}_{n=0}^{\infty} \  \cup \ \{x_{k_n}\}_{n=p}^{\infty} $$ where $p$ satisfies $k_p>i_0>k_{p-1}\ .$

\begin{lemma}\label{lemma5} Consider the two sequences $\{u_{n}\}_{n=0}^{\infty}$ and $\{v_{n}\}_{n=0}^{\infty}$ generated by Algorithm \ref{al1}. Under Assumption \ref{ass}, the following claims hold true.\\
i) The sequence \(\{u_n\}_{n=0}^{\infty}\) (resp. \(\{v_n\}_{n=0}^{\infty}\)) lies within the interval $[x_{i_0},y)$ (resp. $(x^*,x_{k_0}]$). 
Moreover, \(\lim\limits_{n\to\infty}u_n= y\) and  \(\lim\limits_{n\to\infty}v_n=x^*\). \\
ii) For any large enough $n$, \(|f(v_{i_n})|>|f(u_{i_n})|\) and \(|f(v_{k_n})|<|f(u_{k_n})|\).\\
iii) The point $y$ is also a zero of $f$, i.e., $f(y)=0$.
\end{lemma}
\begin{proof}
i) Observe that at the n-th step, \(u_n,v_n\) are some elements of the sequence \(\{x_n\}_{n=0}^{\infty}\) and either \(x_n=u_{n+1}\) or \(x_n=v_{n+1}\). If at some \(n_0\)-th step, both \(u_{n_0}\) and \(v_{n_0}\) belong to the same one of intervals \([x_{i_0},y)\) and \((x^*,x_{k_0}]\) such as \((x^*,x_{k_0}]\) then \(x_n\in (x^*,x_{k_0}]\) for all $n>n_0$, which is a contradiction. 
So, \(u_n\in [x_{i_0},y) \) and \(v_n\in (x^*,x_{k_0}]\) for all $n\geq i_0$. Moreover,  we see that $\{u_n\}_{n=0}^{\infty}$ is not decreasing and has $\{x_{i_n}\}_{n=0}^{\infty}$ as a subsequence. Thus, \(\lim\limits_{n\to\infty}u_n= y\). Similarly, \(\lim\limits_{n\to\infty}v_n=x^*\).\\


ii) Firstly, due to the fact that $\lim\limits_{n\to\infty}x_{i_n} \!=\!\lim\limits_{n\to\infty}u_{i_n}\!=\!y$, $\lim\limits_{n\to\infty}x_{k_n}\!=\!\lim\limits_{n\to\infty}v_{k_n}\!=\!x^*$, we see that for large enough $n$, the following inequality holds.
\[ |x_{i_n}-u_{i_n}| < x^*-y < |x_{i_n}-v_{i_n}|.
\]
This inequality, together with the identity $\bigg|\cfrac{f(v_{i_n})}{f(u_{i_n})}\bigg|=\cfrac{|x_{i_n}-v_{i_n}|}{|x_{i_n}-u_{i_n}|}$\ , imply \linebreak that 
\(|f(v_{i_n})|>|f(u_{i_n})|\). Similarly, we have \(|f(v_{k_n})|<|f(u_{k_n})|\).

iii) Since ii) taking the limits as $n\to\infty$\ , we obtain that \(|f(x^*)|\geq|f(y)|\) and \(|f(x^*)|\leq|f(y)|\). This implies $|f(y)|=|f(x^*)|=0$, i.e., $f(y)=0$.   
\end{proof}

From Lemma \ref{lemma5}, since both sequences \(\{f(x_{i_n})\}_{n=0}^{\infty}\) and $\{f(x_{k_n})\}_{n=0}^{\infty}$ converge to zero, we have
\begin{equation}
\label{15}
	\lim\limits_{n\to\infty}|f(x_n)|=0.
\end{equation}\\
Now, we point out a sub-sequence of the sequence \(\{x_n\}_{n=0}^{\infty}\) such that its image under the function \(|f|\) is strictly increasing.
Firstly, denote by \( \chi:=\chi_{(x^*,x_{k_0}]}\) the characteristic function of the interval \((x^*,x_{k_0}]\), i.e.,
\begin{equation*}
\chi(x)=\left\{\begin{array}{ll}
1 & \text{ for all } x\in(x^*,x_{k_0}]\text{ ,}\\
0 & \text{ otherwise.}
\end{array}\right.
\end{equation*} 
Because there are infinitely many \(x_n\) in each of two intervals \([x_{i_0},y)\) and  \((x^*,x_{k_0}]\), there are infinitely many \(x_n\) such that each of the following two cases occurs. 
\begin{itemize}
	\item[i)] \(\chi(x_n)-\chi(x_{n+1})=1\), i.e., $x_n\in(x^*,x_{k_0}]$ and $x_{n+1}\in[x_{i_0},y)$.
	\item[ii)] \(\chi(x_n)-\chi(x_{n+1})=-1\), i.e., $x_n\in[x_{i_0},y)$ and $x_{n+1}\in(x^*,x_{k_0}]$.
\end{itemize}
Now, we call \(n\) a jump if \(\chi(x_n)-\chi(x_{n+1})\neq0\). Let $p$, $q$ be two consecutive jumps such that
\(\chi(x_q)-\chi(x_{q+1})=1\) and \(\chi(x_p)-\chi(x_{p+1})=-1\). For large enough \(p,q\), we consider two cases where \(p>q\) and \(p<q\).\\
\begin{figure}[H]
	\definecolor{xdxdff}{rgb}{0.49019607843137253,0.49019607843137253,1}
	\definecolor{ududff}{rgb}{0.30196078431372547,0.30196078431372547,1}
	\begin{tikzpicture}[line cap=round,line join=round,>=triangle 45,x=1cm,y=1cm]
	\draw [line width=2pt] (-7,-1)-- (5,-1);
	\draw (-4.76,-1.2) node[anchor=north west] {$\ldots$};
	\begin{scriptsize}
	\draw [fill=ududff] (-7,-1) circle (2.5pt);
	\draw[color=red] (-7,-1.3) node {$x_{i_0}$};
	\draw [fill=ududff] (5,-1) circle (2.5pt);
	\draw[color=red] (5,-1.3) node {$x_{k_0}$};
	\draw [fill=xdxdff] (-2,-1) circle (2.5pt);
	\draw[color=red] (-2,-1.3) node {$y$};
	\draw [fill=xdxdff] (-0.33333333333333304,-1) circle (2.5pt);
	\draw[color=red] (-0.333,-1.3) node {$x^*$};
	\draw [fill=xdxdff] (-5.74,-1) circle (2.5pt);
	\draw[color=red] (-5.74,-1.3) node {$x_{q\!+\!1}$};
	\draw [fill=xdxdff] (-3.9,-1) circle (2.5pt);
	\draw[color=red] (-3.9,-1.3) node {$x_p$};
	\draw [fill=xdxdff] (1.28,-1) circle (2.5pt);
	\draw[color=red] (1.28,-1.3) node {$x_{p\!+\!1}$};
	\draw [fill=xdxdff] (2.82,-1) circle (2.5pt);
	\draw[color=red] (2.82,-1.3) node {$x_q$};
	\draw [fill=xdxdff] (-5,-1) circle (2.5pt);
	\draw[color=red] (-5,-1.3) node {$x_{q\!+\!2}$};
	\end{scriptsize}
	\end{tikzpicture}
	\caption{In case $p>q$}
\end{figure}
In case $p>q$, due to $x_{q\!+\!1},\ldots,x_p \in [x_{i_0},y)$, it follows by i) of Lemma 5 that $x_m\!=\!u_{m\!+\!1}$ for all $m=q+1,\ldots,p$. So, $v_m=v_{m\!+\!1}$ for all $m=\!q\!+\!1,\ldots,p$, i.e., $v_{q\!+\!1}\!=\!v_{q\!+\!2}=\!\ldots=\!v_{p\!+\!1}$. On the other hand, $x_q \in (x^*,x_{k_0}]$, so $x_q=v_{q\!+\!1}$. Since ii) of Lemma 5, for large enough $p$ we have $|f(u_{p\!+\!1})|>|f(v_{p\!+\!1})|$. Therefore, 
%
\[ |f(x_p)|\!=\!|f(u_{p\!+\!1})|>|f(v_{p\!+\!1})|\!=\!|f(v_p)|\!=\! \ldots \!=\!|f(v_{q\!+\!1})| \!=\!|f(x_q)|. \\
\]
\begin{figure}[H]
	\definecolor{xdxdff}{rgb}{0.49019607843137253,0.49019607843137253,1}
	\definecolor{ududff}{rgb}{0.30196078431372547,0.30196078431372547,1}
	\begin{tikzpicture}[line cap=round,line join=round,>=triangle 45,x=1cm,y=1cm]
	\draw [line width=2pt] (-7,-1)-- (5,-1);
	\draw (2.44,-1.2) node[anchor=north west] {$\ldots$};
	\begin{scriptsize}\draw [fill=ududff] (-7,-1) circle (2.5pt);
	\draw[color=red] (-7.02,-1.3) node {$x_{i_0}$};
	\draw [fill=ududff] (5,-1) circle (2.5pt);
	\draw[color=red] (4.96,-1.3) node {$x_{k_0}$};
	\draw [fill=xdxdff] (-2,-1) circle (2.5pt);
	\draw[color=red] (-2.02,-1.3) node {$y$};
	\draw [fill=xdxdff] (-0.33333333333333304,-1) circle (2.5pt);
	\draw[color=red] (-0.34,-1.3) node {$x^*$};
	\draw [fill=xdxdff] (-5.24,-1) circle (2.5pt);
	\draw[color=red] (-5.28,-1.3) node {$x_p$};
	\draw [fill=xdxdff] (-4,-1) circle (2.5pt);
	\draw[color=red] (-4,-1.3) node {$x_{q\!+\!1}$};
	\draw [fill=xdxdff] (1.28,-1) circle (2.5pt);
	\draw[color=red] (1.24,-1.3) node {$x_q$};
	\draw [fill=xdxdff] (3.44,-1) circle (2.5pt);
	\draw[color=red] (3.4,-1.3) node {$x_{p\!+\!1}$};
	\draw [fill=xdxdff] (2.04,-1) circle (2.5pt);
	\draw[color=red] (2.08,-1.3) node {$x_{q\!-\!1}$};
	\end{scriptsize}
	\end{tikzpicture}
	\caption{In case $p<q$}
\end{figure}
Analoguously, if $q>p$ then 
\[ |f(x_q)|=|f(v_{q\!+\!1})|>|f(u_{q\!+\!1})|=|f(u_q)|=\ldots = |f(u_{p\!+\!1})| =|f(x_p)|.
\]
To sum up, if \(q<p\) are two consecutive jumps then \(|f(x_q)|<|f(x_p)|\). Let us denote by \(\{j_n\}_{n=0}^{\infty}\) the sequence of jumps. Since some large enough $n_0$, the sequence \(\{|f(x_{j_n})|\}_{n=n_0}^{\infty}\) is strictly increasing. This contradicts to \eqref{15}, and therefore, implies that Assumption \ref{ass} is wrong. This completes the proof of Theorem \ref{thm1}.

\section{Conclusion}
In this paper, we prove the convergence of Regula Falsi method for all continuous functions. In the future, it is of great interest to provide theoretical proofs about the convergence rate of the method. 

\vfill\eject
\end{document}